\documentclass[12pt]{article}
\usepackage{amsmath,amsthm,amssymb,hyperref, color}
\usepackage{fullpage}

\newtheorem{thm}{Theorem}[section]
\newtheorem{claim}[thm]{Claim}
\newtheorem{lem}[thm]{Lemma}
\newtheorem{define}[thm]{Definition}
\newtheorem{cor}[thm]{Corollary}

\def\L{{\mathcal L}}

\def\F{{\mathbb{F}}}

\def\R{{\mathbb{R}}}

\def\C{{\mathbb{C}}}

\def\Im{{\mathbf{Im}}}
\def\Ker{{\mathbf{Ker}}}

\newcommand{\ip}[2]{\langle #1,#2 \rangle}

\def\_{\,\,\,\,\,}

\def\poly{\textsf{poly}}

\def\span{\textsf{span}}

\newcommand{\eps}{\epsilon}

\newcommand{\remove}[1]{}

\begin{document}

\title{Tensor rank and dimension expanders}

\author{Zeev Dvir\thanks{Department of Computer Science and Department of Mathematics,
Princeton University.
Email: \texttt{zdvir@princeton.edu}. Research supported by NSF grant DMS-2246682.}}

\date{}
\maketitle

\begin{abstract}
We prove a lower bound on the rank of tensors constructed from families of linear maps that `expand' the dimension of every subspace. Such families, called {\em dimension expanders} have been studied for many years with several known explicit constructions.  Using these constructions we show that one can construct an explicit $[D]\times [n] \times [n]$-tensor with rank at least $(2 - \eps)n$, with $D$ a constant depending on $\eps$. Our results extend to border rank over the real or complex numbers.

\end{abstract} 

\section{Overview}

In the following, $\F$ denotes a field which is either finite or the real/complex numbers.

An $(n_1,n_2,n_3)$-tensor over $\F$ is a function $T : [n_1]\times [n_2]\times [n_3] \rightarrow \F$.  We can form two dimensional `slices' of the tensor $T$  by fixing one of the variables. For example, we denote by $T(a,*,*)$  the $n_2 \times n_3$ matrix obtained by fixing $i_1$ to $a \in \F$. A tensor $T$ has rank one if it is of the form $T(i_1,i_2,i_3) = f_1(i_1)f_2(i_2)f_3(i_3)$ for some functions $f_j : [n_j] \rightarrow \F$. The rank of a general tensor $T$, denoted $R(T)$ is defined to be the minimal $r$ such that $T$ can be written as a sum of $r$ rank-one tensors.\footnote{There are other notions of rank for tensors, all generalizing matrix rank, but the one we defined is typically called just `tensor rank'.} When $\F$ is the real or complex numbers we also define the {\em border rank} of $T$, denoted $\underline R(T)$ to be the smallest $r$ such that there is a sequence of rank $\leq r$ tensors converging to $T$.\footnote{One can define border rank also over fields of positive characteristic as being in the Zariski closure of low rank tensors.} While it is easy to show the {\em existence} of tensors of high rank (as high as $\Omega(n^2)$ for $(n,n,n)$ tensors) it is a longstanding open problem to construct high-rank tensors explicitly. By explicit construction we mean that, there is a deterministic algorithm that, given $n$, runs in time $\poly(n)$ and outputs an $(n,n,n)$-tensor $T_n$ of the required rank (more formally, an explicit construction produces an infinite  {\em family} of tensors, not a single one).   The difficulty in constructing explicit tensors of high rank stems from the fact that  such constructions would lead to lower bounds on arithmetic circuits, a notoriously difficult open problem in theoretical computer science (see \cite{SY-survey} for more details on this connection).

 The currently best known explicit lower bounds on tensor rank are those in \cite{AFT11} and are of the form $3n - O(\log_2(n))$ for an $(n,n,n)$-tensor (over any field $\F$). Their work also includes a lower bound of $2n - 1$ on the rank of an explicit  $(\log_2(n),n,n)$-tensor (for $n$ a power of two). It is not stated in \cite{AFT11}, but one can also slightly modify their Theorem 5.6 to construct a $(C,n,n)$-tensor with rank $(2 - 2^{-c})n$. As far as we can tell, the above mentioned results from \cite{AFT11} do not extend to border rank.

 For border rank the best explicit lower bound is that of \cite{LM19}  and is of the form $2.02 n$ for an explicit  $(n,n,n)$-tensor. To the best of our knowledge, the best  border rank lower bounds for $(D,n,n)$-tensors with constant $D$ are of the form $3n/2$ \cite{Gri78,JaJa78,Str83,Grie86}. In fact, one can get this bound even with $D=3$.

Our main contribution is a new way to construct $(D,n,n)$-tensors of rank (or border rank)  $(2-\eps)n$ with  $D$ a constant only depending on $\eps$. We  show a general way to construct such tensors from small families of linear maps that `expand' or `spread' the dimension of any subspace. The property we need is captured in the following definition.

\begin{define}[dimension-spreading]
Let $\L = \{L_1,\ldots,L_D\}$ be a set of linear maps $L_i : \F^n \rightarrow \F^n$. We say that $\L$ is {\em (s,t)-dimension-spreading } if, for every subspace $U \subset \F^n$ of dimension $\dim(U) \geq s$ we have
$$ \dim\left( \sum_{i=1}^D L_i(U) \right) \geq t.$$ We can  extend the definition to sets of $n \times n$ matrices (thinking of them as linear maps).
\end{define}

To state our main theorem, we need one additional notation. Given a  set of $n \times n$ matrices ${\cal A} = \{A_1,\ldots,A_D\}$ we define the $(D,n,n)$-tensor $T^{\cal A}$ to be
$$ T^{\cal A}(i,j,k) = (A_i)_{jk}. $$ In other words, $T^{\cal A}$ is the tensor composed of the $D$ `slices' given by the matrices $A_1,\ldots,A_D$.
\begin{thm}\label{thm-main}
Let $\F$ be any field. If ${\cal A} = \{A_1,\ldots,A_D\}$ is  a set of  $n\times n$  matrices that is  $(s,t)$-dimension spreading then the $(D,n,n)$-tensor $T^{\cal A}$ has rank at least $n+t - s$. Furthermore, if $\F$ is the real or complex numbers, then the same bound holds for border-rank as well. 
\end{thm}

We prove this theorem in Section~\ref{sec-mainproof}. An interesting special case of this theorem is that of $(\eps n, (1-\eps)n)$-dimension spreading maps with $\eps$ a small constant. In this case, we can take $D$ to depend only on $\eps$ and the rank bound becomes $(2-2\eps)n$. We discuss explicit constructions of such dimension-spreading maps in Section~\ref{sec-dimexp}. We observe  that these can be constructed in a black-box manner from {\em dimension-expanders} for which there are several known explicit constructions \cite{DS07, DW10,LZ08,BY13,FG15,GRX18}. In particular, using the  construction in \cite{BY13} one can get, for any $\eps$ and any sufficiently large $n$ an explicit family of $(\eps n, (1-\eps)n)$-dimension spreading maps over $\F^n$ of size $D = \poly(1/\eps)$.\footnote{We use the notation $\poly(f) = O(f^C)$ with $C$ some absolute constant.
} A nice property of this construction is that it is field independent -- the same construction works over any field and the defining matrices have only zero and one entries.  Hence, we get the following corollary. 
\begin{cor}\label{cor-zeroone}
	For every $\eps > 0$, there is an explicit family of $(\poly(1/\eps),n,n)$-tensors $T^{n,\eps}$  (defined using zeros and ones only) that has rank at least $(2 - \eps)n$ over any field (or border-rank if $\F$ is the real or complex numbers).
\end{cor}

\section{Proof of Theorem~\ref{thm-main}}\label{sec-mainproof}
We start with a simple claim which restates the condition of having low tensor rank.
\begin{claim}\label{cla-ranktospan}
Let $T$ be a $(D,n,n)$-tensor over $\F$ with $R(T) \leq r$. For each $i \in [D]$ let $A_i = T(i,*,*)$ be the $n \times n$ matrix obtained  as a slice of $T$. Then, there exist $r$ rank-one $n \times n$ matrices $E_1,\ldots,E_r$ that simultaneously span all the $A_i$'s. That is,  
$$ \{A_1,\ldots,A_D\} \subset \span\{ E_1,\ldots,E_r \}.$$
\end{claim}
\begin{proof}
Suppose $T = T^1 + \ldots + T^r$ such that $$T^\ell(i,j,k) = f^\ell(i)g^{\ell}(j)h^\ell(k)$$ is a rank-one tensor for all $\ell \in [r]$. We define the rank one matrices 
$E_\ell $ as follows: 
$$ (E_\ell)_{jk} = g^\ell(j)h^\ell(k).$$
Fix some $i \in [D]$ and notice that
$$ (A_i)_{jk} = T(i,j,k) = \sum_{\ell}T^{\ell}(i,j,k) = \sum_{\ell} f^\ell(i)g^{\ell}(j)h^\ell(k) = \sum_{\ell} f^\ell(i) \left(E_\ell\right)_{jk}.$$ 
And so, $A_i$ is in the span of $E_1,\ldots,E_\ell$ as was required.\footnote{In fact, the proof shows that the two conditions are equivalent, but we will not use it.}
\end{proof}

To prove Theorem~\ref{thm-main}, suppose in contradiction that the rank of $T^{\cal A}$ was 
\begin{equation}\label{eq-boundonr}
	R(T^{\cal A})= r < n+t - s
\end{equation}

Let $E_1,\ldots,E_r$ be the rank-one matrices given by Claim~\ref{cla-ranktospan} spanning the matrices $A_1,\ldots,A_D$.

 For every subset $S \subset [r]$ we denote by $$K_S = \bigcap_{\ell \in S} \Ker(E_\ell)$$ and by $$ I_S = \sum_{\ell \in S} \Im(E_\ell).$$ 
Notice that for every set $S$ we have 
\begin{equation}\label{eq-boundonkerim}
	\dim(K_S) \geq n - |S|, \,\,\,\,\, \dim(I_S) \leq |S|.
\end{equation}

\begin{claim}\label{cla-nonspreading}
For any set $S\subset [r]$  and for every map $A_i$, $i \in [D]$ we have $$A_i(K_S) \subset I_{\overline S}.$$
\end{claim}
\begin{proof}
	Since $$ A_i = \alpha_1 E_1 + \ldots + \alpha_r E_r,$$ for some coefficients $\alpha_1,\ldots,\alpha_r \in \F$ we have that, if $v \in K_S$, then 
	$$ A_i(v) = \sum_{\ell \in \overline S} \alpha_\ell E_\ell(v),$$ which is in $I_{\overline S}$ by definition.
\end{proof}

Let  $S = \{1,2,\ldots, n- s\}$. Then 
\begin{equation}\label{eq-barSsize}
	|\bar S| = r - |S| < n+t - s - (n-s)= t.
	\end{equation}
Now, from Claim~\ref{cla-nonspreading}, (\ref{eq-boundonkerim}) and (\ref{eq-barSsize}) we see that the $A_i$'s are not $(s,t)$-dimension-spreading, since $$A_i(K_S) \subset I_{\bar S},$$ $$\dim(K_S) \geq n - |S| = s$$ and
$$\dim(I_{\bar S}) \leq |\bar S| < t. $$ 

To argue about border rank, suppose that $\F$ is the real or complex numbers and suppose that there is a sequence of rank $\leq r$ tensors $T^m$ converging to $T^{\cal A}$. From the proof above, for each $m$ there will be some $s$ dimensional subspace $U^m$ such that
\begin{equation}\label{eq-dimsumUm}
	\dim\left( \sum_{i=1}^D A_i^m(U^m) \right) < t,
\end{equation}
Where $A^m_1,\ldots,A^m_D$ are the $D$ slices of $T^m$ (so that $A^m_i$ converges to $A_i$ in the limit). By compactness of the Grassmannian, we can w.l.o.g assume that the sequence $U^m$ converges to some $s$ dimensional subspace $\hat U$. Now, we can pass to the limit in (\ref{eq-dimsumUm}) to get that 
\begin{equation*}
	\dim\left( \sum_{i=1}^D A_i(\hat U) \right) < t
\end{equation*} (using lower semi-continuity of matrix rank). Hence, we get that that $\hat U$ violates the $(s,t)$-dimension spreading condition for $\cal A$. This completes the proof of Theorem~\ref{thm-main}.

\section{Constructing dimension-spreading maps}\label{sec-dimexp}
In this section we prove Corollary~\ref{cor-zeroone} by describing an explicit construction of an $(\eps n,(1-\eps)n)$-dimension spreading family of constant (depending on $\eps$) size. The construction  will follow in a black box manner from the related (and well studied) notion of {\em dimension expanders}, defined by Barak, Impagliazzo, Shpilka and Wigderson \cite{BISW} as a linear algebraic analog of expander graphs.

\begin{define}[Dimension Expanders]\label{def-dimexp}
Let $\L = \{L_1,\ldots,L_D\}$ be a set of linear maps $L_i : \F^n \rightarrow \F^n$. We say that $\L$ is a {\em $\tau$-dimension-expander } if, for every subspace $U \subset \F^n$ of dimension $\dim(U) \leq n/2$ we have
$$ \dim\left( \sum_{i=1}^D L_i(U) \right) \geq (1+\tau)\dim(U).$$ We can  extend the definition to sets of $n \times n$ matrices (thinking of them as linear maps).
\end{define}

As we will now show, it is quite straight-forward to construct an $(\eps n, (1-\eps)n)$-spreading family from an $\tau$-expanding one by  taking sufficiently long words (compositions $L_{i_1}\circ L_{i_2} \circ ...$) over the original family.  One potential obstacle is that expansion is only guaranteed up to dimension $n/2$. We now show that this is not really an issue and expansion holds also for larger subspaces. For an $n \times n$ matrix $A$ we denote by $A^{*}$ its conjugate transpose (or just the transpose over finite fields). 
\begin{lem}\label{lem-expandlarge}
Let ${\cal A} = \{A_1,\ldots,A_D\}$ be a set of $n\times n$ matrices that is a $\tau$-dimension expander and such that $\cal A$ is closed under taking conjugate transpose and contains the identity. Let $U \subset \F^n$ be a subspace of dimension $\alpha n$ with $1/2 < \alpha < 1$. Then
$$ \dim\left( \sum_{i=1}^D A_i(U) \right) \geq \left(1+\frac{\tau(1- \alpha)}{2}\right)\dim(U).$$ 
\end{lem}
\begin{proof}
Let $U$ be a subspace of dimension $\alpha n$ with $1/2 < \alpha < 1$. Let
$$ V = \sum_{i=1}^D A_i(U).$$ Since we know that $\cal A$ contains the identity, there exists a non negative $\delta \geq 0$ \ such that 
\begin{equation}\label{eq-dimV}
	\dim(V)  = (1+\delta)\dim(U)= (1+\delta)\alpha n.
\end{equation}
Suppose $w \in V^*$, meaning $$ \ip{A_i^* w}{u}= \ip{w}{A_iu}=0$$ for all $u \in U$ and for all $i \in [D]$.  Since $\cal A$ is closed under transpose conjugate, this implies that, if $w \in V^*$ then $\ip{A_iw}{u}=0$ for all $u \in U$ and $i \in [D]$. This means that
\begin{equation*}
	\sum_{i=1}^D A_i(V^*) \subset U^*
\end{equation*}
And so
\begin{equation}\label{eq-dimsumVstar}
	\dim\left(\sum_{i=1}^D A_i(V^*) \right) \leq \dim(U^*) = (1-\alpha)n.
\end{equation}
Combining (\ref{eq-dimV}) and (\ref{eq-dimsumVstar}) and using the dimension expansion property on $V^*$ (which has dimension at most $n/2$) we get that
$$ (1-\alpha)n \geq \dim\left(\sum_{i=1}^D A_i(V^*) \right) \geq (1+\tau)\dim(V^*) = (1+\tau)(1-(1+\delta)\alpha)n.$$
solving for $\delta$ we get that
$$\delta \geq \frac{\tau(1-\alpha)}{(1+\tau)\alpha} \geq \frac{\tau(1-\alpha)}{2}.$$
\end{proof}

We are now ready to show that sufficiently long words over a dimension expander gives a dimension spreading family. We will use the following notation: let ${\cal A} = \{A_1,\ldots,A_D\}$ be a family of $n\times n$ matrices. For an integer $t \geq 1$ we denote by ${\cal A}^{(t)}$ the set of matrices obtained as a product of $ t$ matrices from $\cal A$.
\begin{lem}
Let ${\cal A} = \{A_1,\ldots,A_D\}$ be a set of $n\times n$ matrices that is a $\tau$-dimension expander and such that $\cal A$ is closed under taking conjugate transpose and contains the identity. Let $t > 3\log_2(1/\eps)/\tau$. Then ${\cal A}^{(t)}$ is $(\eps n, (1-\eps)n)$-dimension spreading.
\end{lem}

\begin{proof}
	Let $U = U_0$ be such that $\dim(U_0)\geq \eps n$ and, for all $t \geq 1$ let $$U_t = \sum_{i \in [D]}A_i(U_{t-1}).$$  Notice that as long as  $\dim(U_t) \leq n/2$ we have that $$\dim(U_{t+1}) \geq (1+ \tau)\dim(U_t).$$ Thus, after at most $t_1 \leq \tau^{-1}\log_2(1/\eps)$ steps we have $\dim(U_{t_1}) > n/2.$ Once we are in this regime, we use Lemma~\ref{lem-expandlarge} to argue about expansion. We can divide the interval $(1/2,1-\eps]$ into $\log_2(1/\eps)-1$ intervals $$J_i = (1 - 2^{-i},1 - 2^{-(i+1)}]$$ for $i = 1,\ldots, \log_2(1/\eps)-1$.  By Lemma~\ref{lem-expandlarge}, if $\dim(U_t) = \alpha n$ with $\alpha \in J_i$ then $$\dim(U_{t+1}) \geq (1+\tau/2^{i+2})\alpha n.$$
	Hence, once $\alpha$ enters $J_i$, it will move to the next interval after at most $2/\tau$ steps. Indeed, suppose $\alpha > 1- 2^{-i}$ then
	$$ \alpha(1 + \tau/2^{i+2})^{2/\tau} > (1 - 1/2^{i})(1+1/2^{i+1}) \geq 1 - 1/2^{i+1}.$$ Hence, after an additional $2\log_2(1/\eps)/\tau$ steps we will have $\dim(U_t) \geq (1 - \eps)n$ as required.
\end{proof}

Hence, an explicit construction of an $\tau$-dimension-expander for some constant $\tau>0$ and with a constant number $D$ of matrices will imply an explicit construction of an $(\eps n, (1-\eps)n)$-dimension spreading family of size $D' = \poly(1/\eps)$. This, together with Theorem~\ref{thm-main} (replacing $\eps$ with $\eps/2$) will imply Corollary~\ref{cor-zeroone}. Notice that we can always include all the conjugate transposed (and the identity) with a constant blow up in $D$.

\paragraph{Constructions of Dimension Expanders:} There are three main approaches for constructing dimension expanders. The first, using representation theory, is due to Lubozky and Zelmanov \cite{LZ08} and works over the complex numbers. What they show is that, if $G$ is a finite group generated as an expander by a  set $S$ (i.e., the corresponding Cayley graph is a vertex expander) and if $\rho : G \mapsto U(\C^n)$ is an irreducible  representation, then the set of linear maps $\rho(S)$ is a dimension expander (for some $\tau>0$ depending on the  expansion of the associated Cayley graph). Another approach, yielding dimension expanders over sufficiently large finite fields, is given in  \cite{FG15,GRX18} and relies on ideas from coding/design theory.

Finally, the most general construction is due to Bourgain and Yehudayoff \cite{BY13} and works over any field. It follows from an explicit construction of a constant degree {\em monotone expander}. A degree $D$ monotone graph is a bipartite graph $G$ on vertex set $[n] \times [n]$ that can be partitioned into $D$ {\em monotone matchings}, which are (partial) matchings in which every pair of edges $(i_1,j_1)$ and $(i_2,j_2)$ satisfy $i_1 < i_2 \,\implies j_1<j_2$. If $G$ is also a vertex expander (every set $A \subset [n]$ of at most $n/2$ left vertices  has at least $(1+\tau)|A|$ neighbors) then $G$ is a degree $D$ monotone expander. It was shown in \cite{DW10} (and implicitly used in \cite{DS07}) that a degree $D$ monotone expander  gives a set of $D$ linear maps that are a $\tau$-dimension expander. The linear maps are defined directly from the matchings comprising the graph as follows: Suppose $f : [n] \rightarrow [n] \cup \{\bot\}$ is a partial function representing one of the monotone matchings such that $i_1 < i_2 \implies f(i_1) < f(i_2)$ (in case both are defined). We define the linear map $A_f$ sending $e_i$ to $e_{f(i)}$ if $f(i) \neq \bot$ and to $0$ otherwise, where $e_1,\ldots,e_n$ is the standard basis of $\F^n$. It is not hard to show that, over any field, the resulting family of  $D$ maps $A_f$, taken over all matchings defining $G$, is a dimension-expander whenever the graph $G$ is a monotone expander. 

Unlike with `standard' expander graph (without monotonicity), there is no easy probabilistic argument showing the existence of constant degree monotone expanders. In \cite{DW10} explicit constructions of monotone expanders with slowly growing (but not constant) $D = D(n)$ were given using an analog of the Zig-Zag product. The construction in  \cite{BY13} is the only known construction (or even existence proof) of constant degree monotone expanders and uses a far reaching extension of the Bourgain-Gambourd `machine' \cite{BG07} (used to construct expanding Cayley graphs for finite groups) to the group $SL_2(\R)$. Using the dimension expanders of \cite{BY13}, which give some (unspecified but explicit) constants $D$ and $\tau$ and using Lemma~\ref{lem-expandlarge} we conclude the proof of Corollary~\ref{cor-zeroone}.

\bibliographystyle{alpha}

\bibliography{tensorrank.bib}

\end{document}